\documentclass{article}%
\usepackage{amssymb}
\usepackage{amsfonts}
\usepackage{amsmath}
\usepackage{graphicx}%
\newtheorem{theorem}{Theorem}
\newtheorem{lemma}{Lemma}

\newtheorem{remark}[lemma]{Remark}

\newenvironment{proof}[1][Proof]{\noindent\textbf{#1.} }{\ \rule{0.5em}{0.5em}}

\newcommand{\pth}{}
\input{\pth Macros.tex}
\nc{\C}{\mathfrak{M}}
\begin{document}

\title{Farthest points on most Alexandrov surfaces}
\author{Jo\"{e}l Rouyer
\and Costin V\^{\i}lcu}
\maketitle

\abstract{We study global maxima of distance functions on most Alexandrov surfaces
with curvature bounded below, where {\sl most} is used in the sense of Baire categories.}

\medskip

{\small Math. Subj. Classification (2010): 53C45}



\section{Introduction}

In this article, by \textit{Alexandrov surface}, we always mean a compact
$2$-dimensional Alexandrov space with curvature bounded below by $\kappa$,
without boundary. Closed Riemannian surfaces with Gauss curvature at least
$\kappa$ and $\kappa$-polyhedra are important examples of such surfaces. Other
significant examples, for $\kappa=0$, are convex surfaces (\textit{i.e.},
boundaries of compact convex sets with non-empty interior) in $\mathbb{R}^{3}%
$. Roughly speaking, these surfaces are $2$-dimensional topological manifolds
endowed with an intrinsic metric which verifies Toponogov's comparison
property. For the precise definition and basic properties of Alexandrov
spaces, see for example \cite{BGP}.

Let $\mathcal{A}\left(  \kappa\right)  $ be the set of Alexandrov surfaces
with curvature bounded below by $\kappa$. Before going further, we recall the
following simple fact.

\begin{remark}
\label{k} Multiplying all distances in $A\in\mathcal{A}(\kappa)$ by a constant
$\delta>0$ provides an Alexandrov surface $\delta A\in\mathcal{A}(\frac
{\kappa}{\delta^{2}})$. Moreover, for $A^{\prime}\in\mathcal{A}(\kappa)$, the
Gromov-Hausdorff distance between $\delta A$ and $\delta A^{\prime}$ is
exactly $\delta$ times the Gromov-Hausdorff distance between $A$ and
$A^{\prime}$. Therefore, the spaces $\mathcal{A}(\kappa)$ and $\mathcal{A}%
(\frac{\kappa}{\delta^{2}})$ are homothetic, and we may assume without loss of
generality that $\kappa\in\{-1,0,1\}$.
\end{remark}

Two Alexandrov surfaces belong to the same connected component of
$\mathcal{A}(\kappa)$ if and only if they are homeomorphic to each other
\cite{RV2}. Denote by $\mathcal{A}\left(  \kappa,\chi\right)  $ the set of all
surfaces in $\mathcal{A}(\kappa)$ whose Euler-Poincar\'{e} characteristic is
$\chi$. Then $\mathcal{A}\left(  \kappa,\chi\right)  $ (if non-empty) is a
connected component of $\mathcal{A}\left(  \kappa\right)  $ if $\chi$ is
positive or odd, and is the union of two components otherwise. In particular,
$\mathcal{A}\left(  0\right)  $ has four components, consisting respectively
of flat tori, flat Klein bottles (both of these in $\mathcal{A}\left(
0,0\right)  $), convex surfaces (in $\mathcal{A}\left(  0,2\right)  $), and
non-negatively curved projective planes (in $\mathcal{A}\left(  0,1\right)
$). Sometimes it is necessary to exclude the components consisting of flat
surfaces; we shall then replace $\mathcal{A}\left(  \kappa\right)  $ by
\[
\mathcal{B}\left(  \kappa\right)  =\left\{
\begin{array}
[c]{l}%
\mathcal{A}\left(  \kappa\right) \\
\mathcal{A}\left(  0,1\right)  \cup\mathcal{A}\left(  0,2\right)
\end{array}
\right.
\begin{array}
[c]{l}%
\text{if }\kappa\neq0,\\
\text{if }k=0.
\end{array}
\]

It is known that, endowed with the topology induced by the Gromov-Hausdorff
distance, the set $\mathcal{A}(\kappa)$ is a Baire space \cite{IRV2}. In any
Baire space, one says that \textit{most elements} enjoy, or that a
\textit{typical element} enjoys, a given property if the set of those elements
which do not satisfy it is of first category.

The investigation of typical properties of Alexandrov surfaces from Baire
category viewpoint is very recent, see \cite{AZ_2}, \cite{IRV2}, \cite{RV3};
it generalizes a similar, well-established research direction for convex
surfaces, see \eg
the survey \cite{gw}.

\medskip

The study of farthest points on convex surfaces originated from some questions
of H. Steinhaus, presented by H. T. Croft, K. J. Falconer and R. K. Guy in the
chapter A35 of their book \cite{cfg}. The questions have been answered since
then, mainly by T. Zamfirescu in a series of papers \cite{z-q}, \cite{Z2},
\cite{z-fp}, \cite{z-ep}, and yielded several results about farthest points on
most convex surfaces; see the survey \cite{v2}. A few results have also been
obtained for general Alexandrov surfaces \cite{v3}, \cite{vz2}. The framework
of a typical Riemannian metric was considered in \cite{JR3}.

\medskip

In this paper we employ Baire categories to obtain properties of farthest
points on Alexandrov surfaces, thus contributing to the study proposed by H.
Steinhaus. We generalize results about convex surfaces from \cite{Z2} by showing
that, on most surfaces $A\in\mathcal{B}\left(  \kappa\right)  $, most points
$x\in A$ have a unique farthest point (Theorem \ref{TH2}) which is joined to
$x$ by precisely $3$ segments (Theorem \ref{segments}). In particular, Theorem
\ref{TH2} gives an answer to the last open problem in \cite{z-ep}.

The restriction to $\mathcal{B}(\kappa)$ concerning Theorem \ref{TH2} is
mandatory. Indeed, all points on a typical flat torus have two farthest
points, while in the 
connected component $\mathcal{K} \subset \mathcal{A}\left(  0\right)$
of flat Klein bottles there is no typical behavior.
More precisely, there exists open sets $\mathcal{U}$, $\mathcal{V}$ in $\mathcal{K}$ such that, on any
surface in $\mathcal{U}$ most (but not all) points have two antipodes, while
on any surface in $\mathcal{V}$ there exits open sets of points $U_{1}$ and
$U_{2}$ such that any point of $U_{i}$ has precisely $i$ farthest points,
$i=1,2$. This is proven using only elementary methods \cite{RV_flat}.

The paper is organized as follows. In Section \ref{SAS}, we recall some useful
facts about Alexandrov surfaces. In Section \ref{SNFP}, we investigate the
typical number of farthest points. There are certain similarities with the
original proof in \cite{Z2} but, as the central argument there does not hold
in our framework, we replaced it by a new one derived from the main result in
\cite{JR3}. The last two sections are devoted to the typical number of
segments between a point and its unique farthest point. Once again, the
original argument does not apply in our framework. The change of (the sign of)
the curvature bound forces us to develop a new geometric argument, while the
lack of extrinsic geometry leads us to consider, as an auxiliary space, the
space of all Alexandrov metrics on a given topological surface.


\section{Alexandrov surfaces\label{SAS}}

In this section we give necessary prerequisites and notation for the rest
of the paper. Other basic facts about Alexandrov surfaces, implicitly used in the paper,
can be found in \cite{BGP}, \cite{OS}, \cite{ST}.

For $A\in\mathcal{A}\left(  \kappa\right)  $ and $x\in A$, $\rho_{x}=d\left(
x,\cdot\right)  $ denotes the \textit{distance function from} $x$, and $F_{x}$
stands for the set of all \textit{farthest points from} $x$, that is, the
global maxima of $\rho_{x}$. The set of local maxima of $\rho_{x}$ will be
denoted by $M_{x}$. Moreover, $B\left(  x,r\right)  $ and $\bar{B}\left(
x,r\right)  $ stand respectively for the open and closed ball of radius $r$
centered at $x$.

The above notation assumes that the space $A$ and its metric are clear from
the context. In some cases, however, it will be necessary to specify the
metric space $A$ or the metric $d$ itself (\textsl{e.g.}, when several metrics
on the same space are considered), and we shall add a superscript: $F_{x}^{d}%
$, $M_{x}^{A}$, $\rho_{x}^{A}$, etc.

\medskip

Let $f:X\rightarrow Y$ be a map between metric spaces; the \emph{distortion}
of $f$ is given by
\[
\mathrm{dis}\left(  f\right)  =\sup_{x,x^{\prime}\in X}\left\vert d\left(
x,x^{\prime}\right)  -d\left(  f(x),f(x^{\prime})\right)  \right\vert \text{.}%
\]

\begin{lemma}
[Perel'man's stability theorem]\label{LP} \cit{Per1} Let $A_{n}$,
$A\in\mathcal{A}(\kappa)$ and suppose that there exist functions
$f_{n}:A\rightarrow A_{n}$ such that $\mathrm{dis}\left(  f_{n}\right)
\rightarrow0$. Then, for $n$ large enough, there exist homeomorphisms
$h_{n}:A\rightarrow A_{n}$ such that $\sup_{x\in A}d\left(  f_{n}\left(
x\right)  ,h_{n}\left(  x\right)  \right)  \rightarrow0$.
\end{lemma}

We denote by $d_{H}^{Z}\left(  H,K\right)  $ the usual \emph{Pompeiu-Hausdorff
distance} between compact subsets $H$ and $K$ of a metric space $Z$, and omit
the superscript $Z$ whenever no confusion is possible. If $X$ and $Y$ are
compact metric spaces, we denote by $d_{GH}\left(  X,Y\right)  $ the
\emph{Gromov-Hausdorff distance} between $X$ and $Y$. Therefore $d_{GH}\left(
H,K\right)  \leq d_{H}^{Z}\left(  H,K\right)  $, and a partial converse is
given by the following lemma.

\begin{lemma}
\label{SEL} \cit{JR7} Let $\left\{  X_{n}\right\}  _{n\in\mathbb{N}}$ be a
sequence of compact metric spaces converging to $X$ with respect to the
Gromov-Hausdorff metric, and let $\left\{  \varepsilon_{n}\right\}
_{n\in\mathbb{N}}$ be a sequence of positive numbers. Then there exist a
compact metric space $Z$, an isometric embedding $\varphi:X\rightarrow Z$ and,
for each non-negative integer $n$, an isometric embedding $\varphi_{n}%
:X_{n}\rightarrow Z$, such that
\[
d_{H}^{Z}\left(  \varphi_{n}\left(  X_{n}\right)  ,\varphi\left(  Y\right)
\right)  <d_{GH}\left(  X_{n},X\right)  +\varepsilon_{n}.
\]

\end{lemma}

Consider two surfaces $S$ and $S^{\prime}$ with boundaries $\partial S$ and
$\partial S^{\prime}$, and two arcs $I\subset\partial S$ and $I^{\prime
}\subset\partial S^{\prime}$ having the same length. By \textit{gluing }%
$S$\textit{\ to }$S^{\prime}$\textit{\ along} $I$ we mean identifying the
points $x\in I$ and $\iota(x)\in I^{\prime}$, where $\iota:I\rightarrow
I^{\prime}$ is a length preserving map between $I$ and $I^{\prime}$.

\begin{lemma}
[Alexandrov's gluing theorem]\label{gluing} \cit{Pog} Let $S$ be a closed
topological surface obtained by gluing finitely many geodesic polygons cut out
from surfaces in $\mathcal{A}(\kappa)$, in such a way that the sum of the
angles glued together at each point is at most $2\pi$. Then, endowed with the
induced intrinsic metric, $S$ belongs to $\mathcal{A}(\kappa)$.
\end{lemma}

For any surface $A\in\mathcal{A}\left(  \kappa\right)  $, and any $x\in A$, we
denote by $\Sigma_{x}$ the set of directions at $x\in A$. It is known that
$\Sigma_{x}$ is isometric to a circle of length at most $2\pi$ \cite{BGP}. The
point $x$ is said to be \emph{conical} if the length of $\Sigma_{x}$ is less
than $2\pi$. The \emph{singular curvature} at the point $x$ is defined as
$2\pi$ minus the length of $\Sigma_{x}$.

Let $\mathbb{M}_{\kappa}$ denote the simply-connected and complete Riemannian
surface of constant curvature $\kappa$. A $\kappa$\emph{-polyhedron} is an
Alexandrov surface obtained by gluing finitely many geodesic polygons from
$\mathbb{M}_{\kappa}$.

Denote by $\mathcal{R}\left(  \kappa\right)  $ the set of all closed
Riemannian surfaces with Gauss curvature at least $\kappa$, and by
$\mathcal{P}\left(  \kappa\right)  $ the set of $\kappa$-polyhedra. A formal
proof for the next result can be found, for instance, in \cite{IRV2}.

\begin{lemma}
\label{Approximation} Both $\mathcal{P}\left(  \kappa\right)  $ and
$\mathcal{R}\left(  \kappa\right)  $ are dense in $\mathcal{A}(\kappa)$.
\end{lemma}

A \emph{segment} on the surface $A$ is a shortest path on $A$. The set of all
segments between two points $x, y \in A$ will be denoted by $\mathfrak{S}%
_{xy}$.

The \emph{cut locus} $C(x)$ of a point $x\in A$ is the set of all extremities,
different from $x$, of maximal (with respect to inclusion) segments starting
at $x$. It is known that $C(x)$ is locally a tree with at most countably many
ramifications points \cite{ST}.

The following lemma is a direct consequence of Theorem 2 in \cite{z-cl} and
the main result in \cite{AZ_2}.

\begin{lemma}
\label{EndPoints} For most $A\in\mathcal{B}\left(  \kappa\right)  $ and any
$x\in A$, most points of $A$ belong to $C\left(  x\right)  $ and are joined to
$x$ by a unique segment.
\end{lemma}


\section{Typical number of farthest points\label{SNFP}}

The aim of this section is to prove the following partial generalization of
Theorem 2 in \cite{Z2}.

\begin{theorem}
\label{TH2} On most $A\in\mathcal{B}\left(  \kappa\right)  $, most $x\in A$
have precisely one farthest point.
\end{theorem}

T. Zamfirescu strengthened Theorem 2 in \cite{Z2} by proving that on any
convex surface $S$, the set of those $x\in S$ for which $F_{x}$ contains more
than one point is $\sigma$-porous \cite{z-ep}. This statement is not true for
all Alexandrov surfaces, as one can easily see for the standard projective
plane, but it would be interesting to prove it for most surfaces
$A\in\mathcal{B}\left(  \kappa\right)  $. With this respect, our Theorem
\ref{TH2} gives a partial answer to the last open problem in \cite{z-ep}.

Note that Theorem \ref{TH2} also admits the following variant in the framework
of Riemannian geometry.

\begin{lemma}
\label{LRR}\cit[Theorems 1 and 2]{JR3} Let $S$ be a closed differentiable
manifold and let $\mathcal{G}$ be the set of all $C^{2}$ Riemannian structures
on $S$, endowed with the $C^{2}$ topology. Then, for most $g\in\mathcal{G}$
and most $x\in S$, $x$ admits a unique farthest point with respect to $g$, to
which it is joined by at most three segments.
\end{lemma}

For the proof of Theorem \ref{TH2} we need two more lemmas.

\begin{lemma}
\label{lim_Fx} Let $A_{n}$, $A\in\mathcal{A}(\kappa)$ be isometrically
embedded in some compact metric space $Z$. Assume that $A_{n}\rightarrow A$
with respect to the Pompeiu-Hausdorff distance in $Z$. Let $x_{n}\in A_{n}$
converge to $x\in A$. If $y_{n}\in F_{x_{n}}^{A_{n}}$ and $y\in A$ such that
$y_{n}\rightarrow y$ then $y\in F_{x}^{A}$.
\end{lemma}

\begin{proof}
If $X$ is a compact subset of $Z$, and $x$ belongs to $X$, we denote by
$\varrho_{x}^{X}$ the radius of $X$ at $x$, \ie, the maximal value of the
distance function from $x$ restricted to $X$. We claim that $\varrho$ is
continuous with respect to $x$ and $X$. Let $X$, $X^{\prime}$ be two compact
subsets of $Z$, let $y\in X$ be a farthest point from $x\in X$, let
$x^{\prime}\in X^{\prime}$, and let $y^{\prime}\in X^{\prime}$ be a closest
point to $y$. Then%
\begin{align*}
\varrho_{x}^{X}  &  =d^{X}\left(  x,y\right) \\
&  \leq d^{Z}\left(  x,x^{\prime}\right)  +d^{X^{\prime}}\left(  x^{\prime
},y^{\prime}\right)  +d^{Z}\left(  y^{\prime},y\right) \\
&  \leq d^{Z}\left(  x,x^{\prime}\right)  +\varrho_{x^{\prime}}^{X^{\prime}%
}+d_{H}^{Z}\left(  X,X^{\prime}\right)  \text{.}%
\end{align*}
Exchanging the roles of $X$ and $X^{\prime}$ we get%
\[
\left\vert \varrho_{x}^{X}-\varrho_{x^{\prime}}^{X^{\prime}}\right\vert \leq
d^{Z}\left(  x,x^{\prime}\right)  +d_{H}^{Z}\left(  X,X^{\prime}\right)
\text{,}%
\]
thus proving the claim.

Now $d^{A}\left(  x,y\right)  =\lim d^{A_{n}}\left(  x_{n},y_{n}\right)
=\lim\varrho_{x_{n}}^{A_{n}}=\varrho_{x}^{A}$, whence $y\in F_{x}^{A}$.
\end{proof}

\begin{lemma}
\label{LAR} Let $A\in\mathcal{B}\left(  \kappa\right)  $. Then $A$ can be
approximated by Riemannian surfaces with Gaussian curvature strictly larger than $\kappa$.
\end{lemma}

\begin{proof}
If $\kappa\neq0$, the conclusion is reached by contractions or dilations (see
Remark \ref{k}) and by the density of $\mathcal{R}\left(  \kappa^{\prime
}\right)  $ in $\mathcal{A}(\kappa^{\prime})$. We prove here that each
$A\in\mathcal{A}\left(  0,1\right)  \cup\mathcal{A}\left(  0,2\right)  $ can
be approximated by Riemannian surfaces of positive curvature.

We shall show that $A$ can be approached by $\kappa$-polyhedra of
$\mathcal{A}\left(  \kappa\right)  $, with $\kappa>0$ (but of course tends to
$0$). Then the approximation of such polyhedra by surfaces of $\mathcal{R}%
\left(  \kappa\right)  $ will follow from Lemma \ref{Approximation}.

Suppose first that $A\in\mathcal{A}\left(  0,2\right)  $. By Alexandrov's
realization theorem \cite{al}, $A$ is isometric to a convex surface in
$E\overset{\mathrm{def}}{=}\mathbb{R}^{3}$. This surface can be approached by
convex polyhedra $P\subset E$. Assume that $E$ is embedded in $\mathbb{R}^{4}$
as the affine hyperplane $\{R\} \times\mathbb{R}^{3}$. Let $S$ be the sphere
of radius $R$ centered at $0$, and consider the radial projection $\phi
_{R}:E\rightarrow S$, defined by $\phi_{R}\left(  x\right)  =\frac
{Rx}{\left\Vert x\right\Vert }$, where $\left\Vert x\right\Vert $ denotes the
Euclidean norm of $x$ as a vector in $\mathbb{R}^{4}$. It is clear that
$\phi_{R}$ maps lines of $E$ on great circle of $S$, planes of $E$ on totally
geodesic subspaces of $S$, and consequently $0$-polyhedra of $E$ on $\frac
{1}{R^{2}}$-polyhedra of $S$. Rather obviously too, $\phi_{R}\left(  P\right)
$ tends to $P$ when $R$ tends to infinity. (Indeed, a straightforward
computations shows that $d^{S}\left(  \phi_{R}\left(  a\right)  ,\phi
_{R}\left(  b\right)  \right)  =\left\Vert b-a\right\Vert +o\left(  \frac
{1}{R}\right)  $).

If $A\in\mathcal{A}\left(  0,1\right)  $, then one can apply the same
construction to its universal covering $C$, which is a centrally symmetric
convex surface. Since $\phi_{R}$ obviously preserves this symmetry, $\phi
_{R}\left(  C\right)  $ can be quotiented to obtain an approximation of $A$.
\end{proof}

\medskip

\begin{proof}
[Proof of Theorem \ref{TH2}]For $A\in\mathcal{B}\left(  \kappa\right)  $,
define the following sets:
\[
T\left(  A\right)  \overset{\mathrm{def}}{=}\set(:x\in A|\#F_{x}\geq2:),
\]%
\[
T\left(  A,\varepsilon\right)  \overset{\mathrm{def}}{=}\set(:x\in
A|\mathrm{diam}\left(  F_{x}\right)  \geq\varepsilon:).
\]
We have $T\left(  A\right)  =\bigcup_{k\in\mathbb{N}^{\ast}}T\left(
A,1/k\right)  $ and $T\left(  A,1/k\right)  $ is obviously closed in $A$.
Further define:
\[
\mathcal{M}\overset{\mathrm{def}}{=}\set(:A\in\mathcal{B}\left(
\kappa\right)  |T\left(  A\right)  \text{ is not meager}:)\text{,}%
\]%
\[
\mathcal{M}\left(  \varepsilon\right)  \overset{\mathrm{def}}{=}%
\set(:A\in\mathcal{B}\left(  \kappa\right)  |\mathrm{int}\left(  T \left(
A,\varepsilon\right) \right)  \neq\varnothing:)\text{,}%
\]%
\[
\mathcal{M}\left(  \varepsilon,\eta\right)  \overset{\mathrm{def}}%
{=}\set(:A\in\mathcal{B}\left(  \kappa\right)  |\exists x\in A\;\bar{B}\left(
x,\eta\right)  \subset T\left(  A,\varepsilon\right)  :)\text{,}%
\]
whence
\[
\mathcal{M}=\bigcup_{k\in\mathbb{N}^{\ast}}\mathcal{M}\left(  \frac{1}%
{k}\right)  \subset\bigcup_{\substack{k\in\mathbb{N}^{\ast}\\r\in
\mathbb{N}^{\ast}}}\mathcal{M}\left(  \frac{1}{k},\frac{1}{r}\right)  .
\]

It remains to prove that $\mathcal{M}\left(  \varepsilon,\eta\right)  $ is
nowhere dense in $\mathcal{B}\left(  \kappa\right)  $.

We first show that it is closed. Let $A_{n}$ be a sequence of Alexandrov
surfaces of $\mathcal{M}\left(  \varepsilon,\eta\right)  $ converging to
$A\in\mathcal{B}\left(  \kappa\right)  $. Assume those surfaces embedded in a
same compact metric space (see Lemma \ref{SEL}). Let $x_{n}\in A_{n}$ be such
that $\bar{B}\left(  x_{n},\eta\right)  \subset T\left(  A_{n},\varepsilon
\right)  $. Select a converging subsequence of $\left\{  x_{n}\right\}  $ and
denote by $x$ its limit.

We claim that $\bar{B}\left(  x,\eta\right)  \subset T\left(  A,\varepsilon
\right)  $. Choose $y\in\bar{B}\left(  x,\eta\right)  $, hence $y$ is limit of
$\left\{  y_{n}\right\}  $, with $y_{n}\in\bar{B}\left(  x_{n},\eta\right)
\subset T\left(  A_{n},\varepsilon\right)  $. Consequently, there exist
$u_{n},v_{n}\in F_{y_{n}}\subset A_{n}$ such that $d^{A_{n}}\left(
u_{n},v_{n}\right)  \geq\varepsilon$. By extracting subsequences, one can
assume that $\left\{  y_{n}\right\}  $, $\left\{  u_{n}\right\}  $, and
$\left\{  v_{n}\right\}  $ converge to $y\in\bar{B}\left(  x,\eta\right)  $,
$u\in A$, and $v\in A$. By Lemma \ref{lim_Fx} and the continuity of the metric
function we get $u,v\in F_{y}\subset A$ and $d^{A}\left(  u,v\right)
\geq\varepsilon$. That is, $y\in T\left(  A,\varepsilon\right)  $, the claim
is proved, and $\mathcal{M}\left(  \varepsilon,\eta\right)  $ is closed.

Suppose now that $\mathcal{M}\left(  \varepsilon,\eta\right)  $ has an
interior point $A$. By Lemma \ref{LAR}, there exists a sequence $R_{n}%
\in\mathcal{R}\left(  \kappa_{n}\right)  $ tending to $A$ in $\mathcal{R}%
\left(  \kappa\right)  $ for some sequence of numbers $\kappa_{n}>\kappa$. For
$n$ large enough, $R_{n}$ is also interior to $\mathcal{M}\left(
\varepsilon,\eta\right)  $. Put $\mathcal{R}_{1}=\seti(:R\in\mathcal{R}%
|\text{for most~}x\in R,\#F_{x}=1:)$. By Lemma \ref{LRR}, there exists a
sequence of $R_{n,p}\in\mathcal{R}_{1}$ converging to $R_{n}\in\mathcal{R}%
\left(  \kappa_{n}\right)  $ for the $C^{2}$ topology. Hence for $p$ large
enough, $R_{n,p}\in\mathcal{M}\left(  \varepsilon,\eta\right)  \cap
\mathcal{R}\left(  \kappa\right)  $, which is in contradiction with
$R_{n,p}\in\mathcal{R}_{1}$. Hence $\mathcal{M}\left(  \varepsilon
,\eta\right)  $ has empty interior. This completes the proof of Theorem
\ref{TH2}.
\end{proof}


\section{Two lemmas\label{SLemmas}}

In this section we give two auxiliary results which will be invoked in Section
\ref{segm}; the first one seems to have an interest in its own.

Let $S$ be a closed topological surface. We denote by $\C_{k}(S)$ the space of
continuous metrics $d$ on $S$ such that $(S,d)\in\mathcal{A}\left(
\kappa\right)  $. $\C_{k}(S)$ is naturally endowed with the metric $\delta$
defined by
\[
\delta\left(  d,d^{\prime}\right)  =\max_{x,y\in S}\left\Vert d\left(
x,y\right)  -d^{\prime}\left(  x,y\right)  \right\Vert \text{.}%
\]

\begin{lemma}
\label{Baire_A} The space $\C_{\kappa}(S)$ is Baire.
\end{lemma}

\begin{proof}
Consider in $\mathbb{M}_{\kappa}$ a triangle $xyz$ such that $d\left(
x,y\right)  =a$, $d\left(  x,z\right)  =b$, and $d\left(  y,z\right)  =c$. We
denotes by $\Theta_{\kappa}\left(  a;b,c\right)  $ the angle $\measuredangle
xzy$. Moreover, we set $\Theta_{\kappa}\left(  a;b,c\right)  =0$ whenever
$bc=0$ or the triple $\left(  a,b,c\right)  $ does not satisfy one of the
three triangle inequalities. Hence $\Theta_{\kappa}$ is defined on
$\mathbb{R}_{+}^{3}$ and is lower semi-continuous.

Since $S$ is compact, the set $\mathfrak{F}\left(  S\right)  $ of all
non-negative valued continuous functions on $S\times S$ is complete. Denote by
$\mathfrak{F}_{\kappa}^{\prime}\left(  S\right)  $ the set of those functions
$f\in\mathfrak{F}\left(  S\right)  $ such that%

\begin{align}
\forall x,y\in S  &  \,\,f\left(  x,y\right)  =f\left(  y,x\right)
\label{101}\\
\forall x\in S  &  \,\,f\left(  x,x\right)  =0\label{102}\\
\forall x,y,z\in S  &  \,\,f\left(  x,y\right)  \leq f\left(  x,z\right)
+f\left(  z,x\right) \label{103}\\
\forall x,y\in S  &  \,\,\exists z\in A~f\left(  x,z\right)  =f\left(
z,y\right)  =\frac{1}{2}f\left(  x,y\right) \label{104}\\
\forall x,y,z,p\in S  &  \,\,x\neq p~\text{and}~y\neq p~\text{and}~z\neq
p\Longrightarrow\label{105}\\
&  \Theta_{k}\left(  f\left(  x,y\right)  ;f\left(  p,x\right)  ,f\left(
p,y\right)  \right) \nonumber\\
+  &  \Theta_{k}\left(  f\left(  y,z\right)  ;f\left(  p,y\right)  ,f\left(
p,z\right)  \right) \nonumber\\
+  &  \Theta_{k}\left(  f\left(  z,x\right)  ;f\left(  p,z\right)  ,f\left(
p,x\right)  \right)  \leq\pi\text{.} \nonumber
\end{align}

The set $\mathfrak{F}_{\kappa}^{\prime}\left(  S\right)  $ is obviously closed
in $\mathfrak{F}\left(  S\right)  $ and consequently complete. Now, notice
that (\ref{101}), (\ref{102}), (\ref{103}) are standard axioms of metrics,
(\ref{104}) is well-known to imply that the metric is intrinsic, and
(\ref{105}) is the so-called 4 points property in \cite{BGP}, which is one of
the alternative definitions of Alexandrov spaces with curvature bounded below
by $\kappa$. It follows that a function $f\in\mathfrak{F}_{\kappa}^{\prime
}\left(  S\right)  $ belongs to $\C_{\kappa}\left(  S\right)  $ if and only if
it satisfies furthermore $f\left(  x,y\right)  =0\Rightarrow x=y$. In other
words
\[
\mathfrak{F}_{\kappa}^{\prime}\left(  S\right)  \setminus\C_{\kappa}\left(
S\right)  \subset\bigcup_{n\in\mathbb{N}}\mathfrak{H}_{n}\text{,}%
\]
where
\[
\mathfrak{H}_{n}=\set(:f\in\mathfrak{F}_{\kappa}^{\prime}\left(  S\right)
|\exists x,y\in S~\text{s.t.}~d_{0}\left(  x,y\right)  \geq\frac{1}%
{n}~\text{and}~f\left(  x,y\right)  =0:)
\]
and $d_{0}$ is any fixed distance on $S$. The sets $\mathfrak{H}_{n}$ are
closed in $\mathfrak{F}_{\kappa}^{\prime}\left(  S\right)  $, hence
$\mathrm{cl}\left(  \C_{\kappa}\left(  S\right)  \right)  \setminus\C_{\kappa
}\left(  S\right)  $ is included in a countable union of closed sets in
$\mathrm{cl}\left(  \C_{\kappa}\left(  S\right)  \right)  $. On the other
hand, the interior of $\mathrm{cl}\left(  \C_{\kappa}\left(  S\right)
\right)  \setminus\C_{\kappa}\left(  S\right)  $ is obviously empty, whence
$\C_{\kappa}\left(  S\right)  $ in residual in its closure, which is complete.
Therefore it is a Baire space.
\end{proof}

\medskip

The next elementary result is a small step in the proof of Theorem
\ref{segments}.

\begin{lemma}
\label{01} Let $\kappa$ be $0$ or $1$. Consider in $\mathbb{M}_{\kappa}$ a
non-convex quadrilateral $x_{0}x_{1}zx_{2}$ depending on two parameters $l$
and $\varepsilon$, such that the respective midpoints $y_{1}$ and $y_{2}$ of
$x_{0}x_{1}$ and $x_{0}x_{2}$ satisfy $x_{0}x_{1}=x_{0}x_{2}=2l$,
$x_{1}z=x_{2}z=l$, $y_{1}y_{2}=\varepsilon$, see Figure \ref{Fig2}. In the
case $\kappa=1$, assume moreover that $l<\frac{\pi}{2}$. Then, for
$\varepsilon$ small enough, the circumcenter of the triangle $x_{0}x_{1}x_{2}$
lies inside the triangle $y_{1}y_{2}z$.
\end{lemma}

\begin{proof}
Let $\alpha$, $\beta$, $\theta$, $\phi$ and $\psi$ be the angles defined by
Figure \ref{Fig2}; we have to prove that $\beta>\alpha$. \begin{figure}[ptb]
\centering\includegraphics[width=0.7\textwidth]{\pth 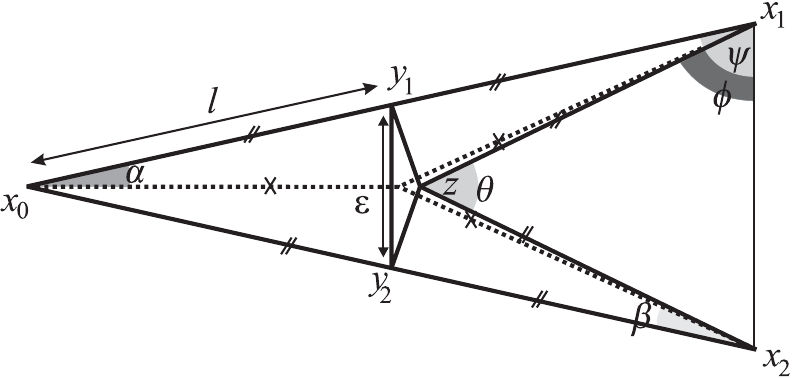} \caption{A
configuration of isosceles triangles in $\mathbb{M}_{\kappa}$.}%
\label{Fig2}%
\end{figure}

In the case $\kappa=0$, we may assume without loss of generality that $l=1$,
hence $\alpha=\arcsin\frac{\varepsilon}{2}$ and $\theta=2\arcsin\varepsilon$.
It follows that
\[
\beta-\alpha=\phi-\psi-\alpha=\frac{\pi}{2}-\alpha-\frac{\pi-\theta}{2}%
-\alpha=\frac{\theta}{2}-2\alpha>0\text{.}%
\]

Now we assume $\kappa=1$. A cosine law in the triangle $x_{0}x_{1}x_{2}$
gives
\[
\cos\varepsilon=\cos^{2}l+\sin^{2}l\cos2\alpha\text{,}%
\]
whence%
\[
\alpha=\frac{1}{2\sin l}\varepsilon\left(  1+\frac{1}{24\tan^{2}l}%
\varepsilon^{2}\right)  +o\left(  \varepsilon^{4}\right)  \text{.}%
\]

Let $a$ be the distance between $x_{1}$ and $x_{2}$. In the triangle
$x_{0}x_{1}x_{2}$, a cosine law gives%
\[
\cos a=\cos^{2}2l+\sin^{2}2l\cos2\alpha\text{,}%
\]
whence%
\[
a=2\cos l\varepsilon+\frac{2\cos l+\cos3l}{12}\varepsilon^{3}+o\left(
\varepsilon^{4}\right)  \text{.}%
\]
A second cosine law in the same triangle gives%
\[
\cos2l=\cos2l\cos a+\sin2l\sin a\cos\phi\text{,}%
\]
whence%
\[
\phi=\frac{\pi}{2}-\frac{\cos l}{\tan2l}\varepsilon-\frac{\cos^{2}l\cos
2l}{48\sin^{3}l}\left(  1+8\sin^{2}l\right)  \varepsilon^{3}+o\left(
\varepsilon^{4}\right)  \text{.}%
\]
In the triangle $x_{1}x_{2}z$ we have%
\[
\cos l=\cos l\cos a+\sin l\sin a\cos\psi\text{,}%
\]
therefore%
\[
\psi=\frac{\pi}{2}-\frac{\cos l}{\tan l}\varepsilon-\frac{\cos^{2}l}%
{24\sin^{3}l}\left(  2\sin^{2}2l+5\cos^{2}l-1\right)  \varepsilon^{3}+o\left(
\varepsilon^{4}\right)  \text{.}%
\]
Finally,%
\begin{align*}
\beta-\alpha &  =\phi-\psi-\alpha\\
&  =\frac{\cos^{2}l}{8\sin^{3}l}\varepsilon^{3}+o\left(  \varepsilon
^{4}\right)  >0\text{.}%
\end{align*}

\end{proof}


\section{Typical number of segments\label{SNS}}

\label{segm}

In this section we generalize the last part of Theorem 2 in \cite{Z2}, from
convex surfaces to Alexandrov surfaces. Results in the same direction were
obtained by the first author \cite{JR3}, for a manifold endowed with a typical
Riemannian metric (see Lemma \ref{LRR}), and by P. Horja \cite{h} and T.
Zamfirescu \cite{z-nr}, for upper curvature bounds.

\begin{theorem}
\label{segments} For most $A\in\mathcal{A}\left(  \kappa\right)  $, most $x\in
A$ and are joined to each of their farthest point by exactly $3$ segments.
\end{theorem}

The proof of our last theorem consists of Lemmas \ref{LFT}, \ref{T2} and
\ref{T3}.

\medskip

The next lemma follows from an explicit computation of cut loci of flat
surfaces \cite{RV_flat}.

\begin{lemma}
\label{LFT} On most $A\in\mathcal{A}(0,0)$, most points are joined to any of
their farthest point by exactly three segments.
\end{lemma}

\begin{lemma}
\label{T2}For most $A\in\mathcal{B}\left(  \kappa\right)  $, for most $x\in
A$, $x$ is joined to each of its farthest points by at most three segments.
\end{lemma}

\begin{proof}
For $x,y\in A$, recall that $\mathfrak{S}_{xy}$ denotes the set of segments
from $x$ to $y$. For $A\in\mathcal{B}\left(  \kappa\right)  $ and
$\varepsilon>0$, define
\[
Q\left(  A\right)  \overset{\mathrm{def}}{=}\set(:x\in A|\exists y\in
F_{x}~\text{with}~\#\mathfrak{S}_{xy}\geq4:)\text{,}%
\]%
\[
Q\left(  A,\varepsilon\right)  \overset{\mathrm{def}}{=}\set(:x\in A|%
\begin{array}
[c]{l}%
\exists y\in F_{x}\;\exists\sigma^{1},\ldots,\sigma^{4}\in\Sigma_{xy}\\
\forall1\leq i<j\leq4\;d_{H}^{A}\left(  \sigma^{i},\sigma^{j}\right)
\geq\varepsilon
\end{array}
:)\text{.}%
\]
Note that $Q\left(  A,\varepsilon\right)  $ is closed, and $Q\left(  A\right)
=\bigcup_{p\in\mathbb{N}^{\ast}}Q\left(  A,1/p\right)  $. Further define
\[
\mathcal{N}\overset{\mathrm{def}}{=}\set(:A\in\mathcal{B}\left(
\kappa\right)  |Q\left(  A\right)  \text{ is not meager}:)\text{,}%
\]%
\[
\mathcal{N}\left(  \varepsilon\right)  \overset{\mathrm{def}}{=}%
\set(:A\in\mathcal{B}\left(  \kappa\right)  |\mathrm{int}\left(  Q\left(
A,\varepsilon\right)  \right)  \neq\varnothing:)\text{,}%
\]%
\[
\mathcal{N}(\varepsilon,\eta)\overset{\mathrm{def}}{=}\set(:A\in
\mathcal{B}\left(  \kappa\right)  |\exists x\in A~\bar{B}\left(
x,\eta\right)  \subset Q\left(  A,\varepsilon\right)  :),
\]
whence
\[
\mathcal{N}=\bigcup_{p\in\mathbb{N}^{\ast}}\mathcal{N}\left(  1/p\right)
\subset\bigcup_{p,q\in\mathbb{N}^{\ast}}\mathcal{N}\left(  1/p,1/q\right)
\text{.}%
\]

We prove now that $\mathcal{N}(\varepsilon,\eta)$ is closed. Let $\{ A_{n} \}
_{n}$ be a sequence of surfaces in $\mathcal{N}(\varepsilon,\eta)$ converging
to $A\in\mathcal{B}\left(  \kappa\right)  $. By Lemma \ref{SEL}, we can assume
that $A_{n}$ and $A$ are embedded in the same compact metric space $Z$. By
hypothesis, there exists $x_{n}\in A_{n}$ such that $\bar{B}\left(  x_{n}%
,\eta\right)  \subset Q\left(  A_{n},\varepsilon\right)  $. We can assume that
$\{ x_{n} \}_{n} $ converges to $x\in A$. A point $z\in\bar{B}\left(
x,\eta\right)  $ is limit of a sequence of points $z_{n}\in Q\left(
A_{n},\varepsilon\right)  $. Hence there exist four segments $\sigma_{n}^{1}$,
\ldots, $\sigma_{n}^{4}$ between $z_{n}$ and $y_{n}\in F_{z_{n}} $. By
selecting a subsequence, we can assume that $\{y_{n}\}_{n} $ converges to
$y\in A$. By Lemma \ref{lim_Fx}, $y\in F_{z}$. We can also assume that $\{
\sigma_{n}^{i}\}_{n} $ converges to some segment $\sigma_{n}^{i}$ from $z$ to
$y$, $i=1$, \ldots, $4$. Moreover, $d_{H}^{A}\left(  \sigma^{i},\sigma
^{j}\right)  =\lim d_{H}^{A}\left(  \sigma_{n}^{i},\sigma_{n}^{j}\right)
\geq\varepsilon$. Hence $\bar B\left(  x,\eta\right)  \subset Q\left(
A,\varepsilon\right)  $, and $A\in\mathcal{N}(\varepsilon,\eta)$.

The proof that $\mathcal{N} (\varepsilon,\eta)$ has empty interior is similar
to the proof of the emptiness of the interior of $\mathcal{M}(\varepsilon
,\eta)$ in the proof of Theorem \ref{TH2}.

It follows that $\mathcal{N}$ is meager, whence the conclusion.
\end{proof}

\begin{lemma}
\label{y_conical} Let $S$ be a closed topological surface such that
$\C_{\kappa}\left(  S\right)  $ is non-empty. Choose $d\in\C_{\kappa}\left(
S\right)  $ and $x$, $y\in S$ such that $F_{x}^{d}=\{y\}$. Then there exist
$d_{n}\in\C_{\kappa}\left(  S\right)  $ converging to $d$, and $x_{n},
y_{n}\in S$ converging to $x$ and $y$ respectively, such that $x_{n}$ and
$y_{n}$ are conical points on $(S,d_{n})$ and $F_{x_{n}}^{d_{n}}=\{y_{n}\}$.
\end{lemma}

\begin{proof}
Choose $d^{\prime}$ close to $d$, such that most points in $\left(
S,d^{\prime}\right)  $ have a unique farthest point to which they are joined
by at most three segments, and such that, for any point $z$, most points in
$S$ belong to $C^{d^{\prime}}\left(  z\right)  $ and are joined to $z$ by
unique segments. This is possible by virtue of Lemma \ref{EndPoints}, Lemma
\ref{T2}, Theorem \ref{TH2} and Lemma \ref{LP}. Chose a point $x^{\prime}$
close to $x$, whose unique farthest point $y^{\prime}$ satisfies
$\#\mathfrak{S}_{x^{\prime}y^{\prime}}\leq3$. By Lemma \ref{lim_Fx},
$y^{\prime}$ tends to $y$ when $d^{\prime}\rightarrow d$ and $x^{\prime
}\rightarrow x$.

Take a point $v\in C(y^{\prime})$, distinct from $x^{\prime}$, joined to
$y^{\prime}$ by a unique segment $\gamma_{v}$. Take $w\in C(y)$ close to $v$
and joined to $y$ by precisely two segments $\gamma_{w}^{1}$ and $\gamma
_{w}^{2}$; this is possible, because $C(y)$ has at most countably many
ramifications points, see \cite{ST}. Remove the part $\Delta$ of
$(S,d^{\prime})$ bounded by $\gamma_{w}^{1}\cup\gamma_{w}^{2}$ and containing
$v$, and glue the rest by identifying the two boundary segments. Denote by
$\left(  S^{\prime\prime},d^{\prime\prime}\right)  $ the obtained surface, and
by $f:S\setminus\Delta\rightarrow S^{\prime\prime}$ the canonical surjection.
When $w$ tends to $v$, $\gamma_{w}^{1}$ and $\gamma_{w}^{2}$ both tends to
$\gamma_{v}$ (because $\gamma_{v}$ is the only segment between $y$ and $v$),
and $\left(  S^{\prime\prime},d^{\prime\prime}\right)  $ obviously tends to
$\left(  S,d^{\prime}\right)  $. Clearly $d^{\prime\prime}\left(  f\left(
p\right)  ,f\left(  q\right)  \right)  \leq d^{\prime}\left(  p,q\right)  $
for any $p$, $q\in S$.

We claim that, for $w$ close enough to $v$, $d^{\prime\prime}(x^{\prime
},y^{\prime})=d^{\prime}(x^{\prime},y^{\prime})$. Let $\gamma$ be a segment on
$\left(  S^{\prime\prime},d^{\prime\prime}\right)  $ between $x^{\prime}$and
$y^{\prime}$. If $\gamma$ does not intersect $f\left(  \gamma_{w}^{1}%
\cup\gamma_{w}^{1}\right)  $, then the claim holds. Otherwise, the limit of
$\gamma$, which is a segment of $\left(  S,d^{\prime}\right)  $ between
$x^{\prime}$ any $y^{\prime}$, should intersect $\gamma_{v}$, which is impossible.

It follows that $F_{x^{\prime}}^{\left(  S^{\prime\prime},d^{\prime\prime
}\right)  }=\left\{  y^{\prime}\right\}  $.

By Lemma \ref{LP}, there exists a homeomorphism $h:S\rightarrow S^{\prime
\prime}$ such that $\mathrm{dis}\left(  h\right)  \rightarrow0$, $x^{\ast
}\overset{\mathrm{def}}{=}h^{-1}\left(  x^{\prime}\right)  \rightarrow
x^{\prime}$ and $y^{\ast}\overset{\mathrm{def}}{=}h^{-1}\left(  y^{\prime
}\right)  \rightarrow y^{\prime}$ when $w\rightarrow v$. Define the metric
$d^{\ast}$ on $S$ by $d^{\ast}\left(  p,q\right)  =d^{\prime\prime}\left(
h\left(  p\right)  ,h\left(  q\right)  \right)  $, we have $y^{\ast}$ conical
point for $d^{\ast}$ and $F_{x^{\ast}}^{d^{\ast}}=\left\{  y^{\ast}\right\}  $.

A similar procedure applied to $x$ produces the desired metric.\medskip
\end{proof}

The following lemma is part of Theorem 3 in \cite{gp}, see also Theorem 1 in
\cite{vz1}. The theorem was originally stated for all points $x\in A$ and only
for $F_{x}$, but the proof given in \cite{vz1} holds here too.

\begin{lemma}
\label{F_1} If $A\in\mathcal{A}\left(  1\right)  $ and $x\in A$ such that
$d(x,F_{x})>\pi/2$ then $\#M_{x}=\#F_{x}=1$.
\end{lemma}

A \textit{long loop} at $x\in A$ is the union of two segments from some $y\in
F_{x}$ to $x$, whose directions at $y$ divide $\Sigma_{y}$ into two arcs of
length at most $\pi$.

\begin{lemma}
\label{Fix_A} Let $S$ be a closed topological surface such that $\C\left(
\kappa\right)  $ is non-empty. For most Alexandrov metrics on $S$, for most
$x\in S$ and any $y\in F_{x}$, there are at least three segments between $x$
and $y$.
\end{lemma}

\begin{proof}
For any metric $d\in\C_{\kappa}\left(  S\right)  $, the set
\[
D\left(  d\right)  \overset{\mathrm{def}}{=}\set(:x\in
S|\mathrm{there\;exists\;a\;long\;loop\;at\;}x\text{ }%
\mathrm{with\;respect\;to\;}d:)
\]
is clearly closed, from the definition of a long loop.

Consider a countable dense subset $Z$ of $S$. We have
\[
\mathcal{D}\overset{\mathrm{def}}{=}\set(:d\in\C_{\kappa}\left(  S\right)
|\mathrm{int}\left(  D(d)\right)  \neq\emptyset:)\subset\bigcup
_{\substack{z\in Z\\q\in\mathbb{N}^{\ast}}}\mathcal{D}_{z,q}\,\text{,}%
\]
where
\[
\mathcal{D}_{z,q}\overset{\mathrm{def}}{=}\set(:d\in\C_{\kappa}\left(
S\right)  |B\left(  z,\frac{1}{q}\right)  \subset D(d):)\text{.}%
\]
Here, the balls $B(z,1/q)$ are understood with respect to any fixed continuous
metric on $S$. The sets $\mathcal{D}_{z,q}$ are clearly closed, it remains to
prove that they have empty interior.

Assume on the contrary that there exist $z\in Z$ and $q\in\mathbb{N}^{\ast} $
such that $\mathrm{int}\left(  \mathcal{D}_{z,q} \right)  \neq\emptyset$; in
the rest of the proof we shall derive a contradiction.

By Lemma \ref{LP}, a dense set in $\mathcal{A}(\kappa)$ corresponds to a dense
set in $\C_{\kappa}(S)$, so it is possible to choose a metric $d^{\prime}%
\in\mathrm{int} \left(  \mathcal{D}_{z,q}\right)  $ such that $(S,d^{\prime})$
is typical in $\mathcal{A}(\kappa)$. By Theorem 3.1 in \cite{IRV2},
$(S,d^{\prime})$ has no conical points, and by Lemma \ref{T2}, we can choose a
point $x^{\prime}\in B(z,\frac{1}{3q})$ which is joined to its unique farthest
point $y^{\prime}$ by at most three segments, two of which are forming a long
loop $\Gamma_{x^{\prime}}^{\prime}$.

The space of directions at each point in $(S,d^{\prime})$ is a circle, hence
we can define locally around $x^{\prime}$ and around $y^{\prime}$ a left and a
right side of $\Gamma_{x^{\prime}}^{\prime}$. We apply twice the procedure
described by Lemma \ref{y_conical} and its proof, one time for each side of
$\Gamma_{x^{\prime}}^{\prime}$. Doing so allows us to assume that, after
cutting, $\Gamma_{x^{\prime}}^{\prime}$ still divide $\Sigma_{x^{\prime}}$ in
two almost equally long curves. Here, `almost equally long' means that the
ratio between the length of the two connected components can be chosen
arbitrary close to one. For the new metric $d\in\mathrm{int}\left(
\mathcal{D}_{z,q} \right)  $, there are conical points $x\in B\left(
z,\frac{1}{2q}\right)  $, $y\in S$ such that $F_{x}=\left\{  y\right\}  $. A
long loop $\Gamma_{x}$ at $x$ through $y$ divides $\Sigma_{y}$ into two
equally long curves. Let $\gamma^{1}$, $\gamma^{2}\in\mathfrak{S}_{xy}$ be the
two segments from $x$ to $y$ composing $\Gamma_{x}$. We may assume, moreover,
that the singular curvatures at $x$ and $y$ verify $\omega_{x}\leq\omega_{y}$.

\medskip

\textbf{Case 1:} $\kappa\in\{-1,0\}$.

Consider in $\mathbb{M}_{0}$ a quadrilateral $L=x_{0}x_{1}zx_{2}$ defined as
in Lemma \ref{01} (see Figure \ref{Fig2}) with $l=d(x,y)$ and $\varepsilon$
small enough to ensure that $2(\alpha+\beta)\leq\omega_{x}$.

Cut $(S,d)$ along $\Gamma_{x}$ and insert $L$ as follows: identify $x_{1}z$ to
$x_{2}z$, identify $x_{0}y_{1}$ to one image of $\gamma^{1}$ and $x_{0}y_{2}$
to the other image of $\gamma^{1}$, and identify $x_{1}y_{1}$ to one image of
$\gamma^{2}$ and $x_{2}y_{2}$ to the other image of $\gamma^{2}$. On the
resulting Alexandrov surface $A_{0}=(S_{0},d_{0})$ (see Lemma \ref{gluing}),
denote by $L_{0}$ the image of $L$, and by the same letters $x,y$ the images
of $x$ and $y$. By Lemma \ref{01}, there exists a point $y_{0}$ (the
circumcenter of $x_{0}x_{1}x_{2}$) interior to $L_{0}$ such that
$d_{0}(x,y_{0})>d_{0}(x,z)$, for any $z\in L_{0}$. It follows that
$F_{x}^{d_{0}}=\{y_{0}\}$ and there are three segments from $x$ to $y_{0}$
which, moreover, make angles smaller that $\pi$ at $y_{0}$. Since $y_{0}$ is a
smooth point of $(S,d_{0})$, there is no long loop at $x$.

Using Lemma \ref{LP}, one can define a metric $d_{1}$ on $S$, such that
$\left(  S,d_{1}\right)  =A_{0}$ and such that the point $x_{1}$ corresponding
to $x$ belongs to $B\left(  z,\frac{1}{q}\right)  $. Hence $d_{1}$ does not
belongs to $\mathcal{D}_{z,q}$ and a contradiction is obtained.

\medskip

\textbf{Case 2:} $\kappa=1$.

\textbf{Subcase 2.1:} $d(x,y)<\pi/2$.

The construction is similar to that at Case 1, but this time we take $L$ in
$\mathbb{M}_{1}$.

\textbf{Subcase 2.2:} $d(x,y)=\pi/2$.

Apply Subcase 2.1 to $\left(  S,\lambda d\right)  $, where $\lambda$ is
slightly less than one.

\textbf{Subcase 2.3:} $d(x,y)>\pi/2$. This cannot be done similarly to Subcase
2.1, because now the isosceles triangles $x_{0}y_{1}y_{2}$, $x_{1}y_{1}z$ and
$x_{2}y_{2}z$ have base angles larger that $\pi$.

It is known that $\mathrm{diam}^{d}(S)\leq\pi$ for any $d$ such that $\left(
S,d\right)  \in\mathcal{A}\left(  1\right)  $ \cite[Theorem 3.6]{BGP}. By
replacing $d$ by $\lambda d$ (with $\lambda<1$ close to $1$), we may assume
without loss of generality that $\mathrm{diam}^{d}(S)<\pi$.

Consider the sphere $\Lambda$ of radius $r\overset{\mathrm{def}}{=}%
d(x,F_{x})/\pi<1$; put $\kappa^{\prime}=\frac{1}{r^{2}}$. Then $\Lambda
\in\mathcal{A}(\kappa^{\prime},2)\subset\mathcal{A}(1,2)$. Consider a slice
$\Theta$ of $\Lambda$ determined by great half-circles making an angle of
$\phi$ at their intersection point, such that $2\phi=\omega_{y}$.

Cut $(S,d)$ along $\Gamma_{x}$ and insert two copies of $\Theta$. Denote by
$A_{1}$ the resulting Alexandrov surfaces (see Lemma \ref{gluing}) and by
$\Theta_{1}$, $\Theta_{2}\subset A_{1}$ the two copies of $\Theta$ on $A_{1}$.
Figure \ref{Fig1} illustrates the approximation procedure for this subcase.

\begin{figure}[ptb]
\centering\includegraphics[width=0.9\textwidth]{\pth 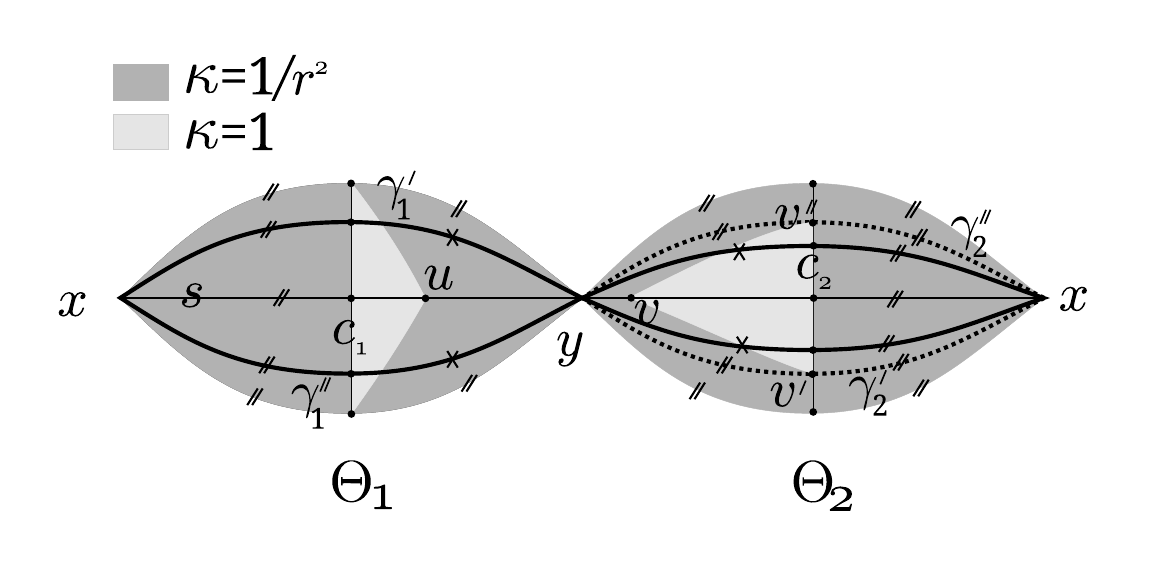}\caption{Approximation
procedure in the proof of Lemma \ref{F_1}.}%
\label{Fig1}%
\end{figure}

Denote by $s_{i}$ the median great half-circle of $\Theta_{i}$ ($i=1,2$). By
considering only subslices of $\Theta_{1}$, $\Theta_{2}$, we may assume
without loss of generality that, on the one hand, the angle at $y$ between the
third segment to $x$ (if it exists) and any direction toward $\Theta_{1}%
\cup\Theta_{2}$ is distinct from $\pi$, and on the other hand, the angle at
$y$ between $s_{1}$ and $s_{2}$ is exactly $\pi$. Put $s\overset{\mathrm{def}%
}{=}s_{1}\cup s_{2}$ and denote by $c_{1}$, $c_{2}$ the arcs of great circles
in $\Theta_{1}$, $\Theta_{2}$ orthogonal to $s$, through the midpoints of
$s_{1}$ and $s_{2}$ respectively. Consider a point $u\in s$ close to $c_{1}$
and replace the triangle whose vertices are $u$ and the endpoints of $c_{1}$
by a triangle of constant curvature $1$ with the same side lengths. After
this, the distances on $\Theta_{1}$ slightly decrease and, by Toponogov's
comparison property, $u$ becomes a conical point. Due to the symmetry of
$\Theta_{1}$ with respect to $s$, on the resulting surface $A_{2}$ there are
two segments from $y$ to $x$, say of length $l_{1}$. Denote them by
$\gamma_{1}^{\prime},\gamma_{1}^{\prime\prime}$ and denote by $\gamma
_{2}^{\prime},\gamma_{2}^{\prime\prime}$ the segments on $\Theta_{2}$ of
tangent directions at $y$ opposite to those of $\gamma_{1}^{\prime}$,
$\gamma_{1}^{\prime\prime}$. Let $v^{\prime}$, $v^{\prime\prime}$ be the
intersection points of $c_{2}$ with $\gamma_{2}^{\prime},\gamma_{2}%
^{\prime\prime}$, and consider a variable point $v\in s_{2}$. Replace the
triangle $vv^{\prime}v^{\prime\prime}$ by a triangle of constant curvature
$1$. Then, on the new surface $A_{2}=A_{2}\left(  v\right)  $, there are two
shortest path on $\Theta_{2}$ from $y$ to $x$, of length $l_{2}\left(
v\right)  $. Note that $l_{2}\left(  c_{2}\cap s_{2}\right)  =\pi r$, and
$l_{2}\left(  y\right)  =(1+r)\pi/2$. Hence there exists $v_{0}\in s$ such
that $l_{2}\left(  v_{0}\right)  =l_{1}$.

On $A_{2}\left(  v_{0}\right)  $ we have four segments from $x$ to $y$ such
that no two of them are composing a loop, and the angle at $y$ between any two
consecutive segments is less than $\pi$. By virtue of the first variation
formula \cite[Theorem 3.5]{OS}, $y$ is a strict local maximum for $\rho
_{x}^{A_{2}\left(  v_{0}\right)  }$, and Lemma \ref{F_1} shows that
$F_{x}=M_{x}=\{y\}$.

Now, using Lemma \ref{LP}, one can define a metric $d_{2}$ on $S$ such that
$\left(  S,d_{2}\right)  =A_{2}\left(  v_{0}\right)  $, and such that $x\in
B\left(  z,\frac{1}{q}\right)  $.

Therefore, the Alexandrov metric $d_{2}$ does not belongs to $\mathcal{D}%
_{z,q}$, and a contradiction is obtained.
\end{proof}

\begin{lemma}
\label{T3}For most $A\in\mathcal{B}\left(  \kappa\right)  $, for most $x\in A
$ and any $y\in F_{x}$, there are at least three segments between $x$ and $y$.
\end{lemma}

\begin{proof}
For any surface $A\in\mathcal{B}\left(  \kappa\right)  $, the set
\[
L\left(  A\right)  \overset{\mathrm{def}}{=}\set(:x\in
A|\mathrm{there\;exists\;a\;long\;loop\;at\;}x:)\text{,}%
\]
is closed, from the definition of a long loop and the semi-continuity of angles.

We have
\[
\mathcal{L}=\set(:A\in\mathcal{B}\left(  \kappa\right)  |\mathrm{int}\left(
L(A) \right)  \neq\emptyset:)=\bigcup_{q\in\mathbb{N}^{\ast}}\mathcal{L}%
_{q}\text{,}%
\]
where
\[
\mathcal{L}_{q}\overset{\mathrm{def}}{=}\set(:A\in\mathcal{B}\left(
\kappa\right)  |\exists x\in A\;\bar{B}\left(  x,\frac{1}{q}\right)  \subset
L(A):).
\]
The sets $\mathcal{L}_{q}$ are clearly closed, and they have empty interior by
Lemma \ref{Fix_A}, hence $\mathcal{L}$ is of first category.
\end{proof}


\bigskip

\noindent\textbf{Acknowledgement.} The authors were partly supported by the
grant PN-II-ID-PCE-2011-3-0533 of the Romanian National Authority for
Scientific Research, CNCS-UEFISCDI.

They express thanks to Tudor Zamfirescu for suggesting them to investigate
properties of most Alexandrov surfaces.


{\small \bigskip}

{\small Jo\"el Rouyer }

{\small \noindent Institute of Mathematics ``Simion Stoilow'' of the Romanian
Academy, \newline P.O. Box 1-764, Bucharest 70700, ROMANIA \newline
Joel.Rouyer@ymail.com, Joel.Rouyer@imar.ro }

{\small \medskip}

{\small Costin V\^{\i}lcu }

{\small \noindent Institute of Mathematics ``Simion Stoilow'' of the Romanian
Academy, \newline P.O. Box 1-764, Bucharest 70700, ROMANIA \newline
Costin.Vilcu@imar.ro }


\begin{thebibliography}{99}                                                                                               %


\bibitem {AZ_2}K. Adiprasito and T. Zamfirescu, \textit{Few Alexandrov
surfaces are Riemann}, J. Nonlinear Convex Anal., \textbf{16}, 1147--1153 (2015).

\bibitem {al}A.D. Alexandrov, \textit{Die innere Geometrie der konvexen
Fl\"{a}chen}, Akademie-Verlag, Berlin, 1955

\bibitem {BGP}Y. Burago, M.~Gromov and G.~Perel'man, \textit{A. D. Alexandrov
spaces with curvature bounded below}., Russ. Math. Surv. \textbf{47} (1992),
1--58 (English. Russian original)

\bibitem {cfg}H. T. Croft, K. J. Falconer and R. K. Guy, \textit{Unsolved
Problems in Geometry}. Springer-Verlag, New York, 1991

\bibitem {gp}K. Grove and P. Petersen, \textit{A radius sphere theorem},
Inventiones Math. \textbf{112} (1993), 577--583

\bibitem {gw}P. Gruber, \textit{Baire categories in convexity}, in P. Gruber
and J. Wills (eds.), \textit{Handbook of Convex Geometry}, vol. B,
North-Holland, Amsterdam, 1993, 1327--1346

\bibitem {h}P. Horja, \textit{On the number of geodesic segments connecting
two points on manifolds of non-positive curvature}, Trans. Amer. Math. Soc.
\textbf{349} (1997), 5021--503

\bibitem {IRV2}J. Itoh, J. Rouyer and C. V{\^{\i}}lcu, \textit{Moderate
smoothness of most Alexandrov surfaces}, Int. J. Math. \textbf{26} (2015), [14 pages]

\bibitem {OS}Y. Otsu and T. Shioya, \textit{The Riemannian structure of
Alexandrov spaces}, J. Differential Geom. \textbf{39} (1994), 629--658

\bibitem {Per1}G. Perel'man, \textit{A. D. Alexandrov spaces with curvatures
bounded from below II}, preprint 1991

\bibitem {Pog}A. V. Pogorelov, \textit{Extrinsic geometry of convex surfaces},
Amer. Math. Soc., 1973

\bibitem {JR3}J. Rouyer, \textit{On antipodes on a manifold endowed with a
generic Riemanniann metric}, Pac. J. Math. \textbf{212} (2003), 187--200

\bibitem {JR7}J. Rouyer, \textit{Generic properties of compact metric spaces},
Topology Appl. \textbf{158} (2011), 2140--2147

\bibitem {RV2}J. Rouyer and C. V{\^{\i}}lcu, \textit{The connected components
of the space of Alexandrov surfaces}, in D. Ibadula and W. Veys (eds.),
\textsl{Experimental and Theoretical Methods in Algebra, Geometry and
Topology}, Springer Proc. in Mathematics and Statistics \textbf{96} (2014), 249--254

\bibitem {RV3}J. Rouyer and C. V\^{\i}lcu, \textit{Simple closed geodesics on
most Alexandrov surfaces}, Adv. Math. \textbf{278} (2015), 103--120

\bibitem {RV_flat}J. Rouyer and C. V\^{\i}lcu, \textit{Cut loci and critical
points on flat surfaces}, manuscript

\bibitem {ST}K. Shiohama and M. Tanaka, \textit{Cut loci and distance spheres
on Alexandrov surfaces}, Actes de la Table Ronde de G\'eom\'etrie
Diff\'erentielle (Luminy, 1992), S\'em. Congr., vol. 1, Soc. Math. France,
Paris, 1996, 531--559

\bibitem {v2}C. V\^\i lcu, \textit{Properties of the farthest point mapping on
convex surfaces}, Rev. Roum. Math. Pures Appl. \textbf{51} (2006), 125--134

\bibitem {v3}C. V\^\i lcu, \textit{Common maxima of distance functions on
orientable Alexandrov surfaces}, J. Math. Soc. Japan \textbf{60} (2008), 51--64

\bibitem {vz1}C. V\^\i lcu and T. Zamfirescu, \textit{Symmetry and the
farthest point mapping on convex surfaces}, Adv. Geom. \textbf{6} (2006), 345--353

\bibitem {vz2}C. V\^{\i}lcu and T. Zamfirescu, \textit{Multiple farthest
points on Alexandrov surfaces}, Adv. Geom. 7 (2007), 83--100

\bibitem {z-q}T. Zamfirescu, \textit{On some questions about convex surfaces},
Math. Nach. \textbf{172} (1995), 313--324

\bibitem {Z2}T. Zamfirescu,\textit{\ Points joined by three shortest paths on
convex surfaces}, Proc. Am. Math. Soc. \textbf{123} (1995), 3513--3518

\bibitem {z-fp}T. Zamfirescu, \textit{Farthest points on convex surfaces},
Math. Z. \textbf{226} (1997), 623--630

\bibitem {z-ep}T. Zamfirescu, \textit{Extreme points of the distance function
on convex surfaces}, Trans. Amer. Math. Soc. \textbf{350} (1998), 1395--1406

\bibitem {z-cl}T. Zamfirescu, \textit{On the cut locus in Alexandrov spaces
and applications to convex surfaces}, Pac. J. Math. \textbf{217} (2004), 375--386

\bibitem {z-nr}T. Zamfirescu, \textit{On the number of shortest paths between
points on ma\-ni\-folds}. Rend. Circ. Mat. Palermo Suppl. \textbf{77} (2006), 643--647
\end{thebibliography}
\end{document}